\documentclass[a4paper,11pt]{amsart}
\usepackage{pb-diagram}
\usepackage[all]{xy}
\usepackage{hyperref}
\usepackage[utf8]{inputenc}
\usepackage{amsxtra}
\usepackage{amsopn}
\usepackage{amsmath,amsthm,amssymb}
\usepackage{amscd}
\usepackage{amsfonts}
\usepackage{latexsym}
\usepackage{verbatim}
\usepackage{color}
\usepackage{bm}
\usepackage{amssymb,latexsym}
\usepackage{enumerate}

\newcommand{\dbar}{\ensuremath{\overline\partial}}

\newcommand\C{{\mathbb C}}

\newcommand\R{{\mathbb R}}
\newcommand\Z{{\mathbb Z}}

\newcommand\Span{{\hbox{\em Span}}}

\newcommand{\del}{\partial}
\newcommand{\delbar}{\overline{\del}}

\makeatletter
\newcommand{\sumprime}{\if@display\sideset{}{'}\sum%
            \else\sum'\fi}
\makeatother

\begin{document}

\numberwithin{equation}{section}

\newtheorem{theorem}{Theorem}[section]
\newtheorem{proposition}[theorem]{Proposition}
\newtheorem{conjecture}[theorem]{Conjecture}
\def\theconjecture{\unskip}
\newtheorem{corollary}[theorem]{Corollary}
\newtheorem{lemma}[theorem]{Lemma}
\newtheorem{observation}[theorem]{Observation}
\newtheorem{definition}[theorem]{Definition}
\newtheorem{remark}[theorem]{Remark}
\def\theremark{\unskip}
\newtheorem{kl}{Key Lemma}
\def\thekl{\unskip}
\newtheorem{question}{Question}
\def\thequestion{\unskip}
\newtheorem{example}{Example}
\def\theexample{\unskip}
\newtheorem{problem}{Problem}


\address{Dipartimento di Scienze Matematiche, Fisiche e Informatiche,
Unit\`{a} di Matematica e Informatica\\
Universit\`{a} degli Studi di Parma\\
Parco Area delle Scienze 53/A, 43124\\
Parma, Italy}
\email{adriano.tomassini@unipr.it}
\address{Department of Mathematical Sciences, Norwegian University of Science and Technology, NO-7491 Trondheim, Norway}
\email{xu.wang@ntnu.no}

\thanks{The first author is partially supported by the Project PRIN ``Varietà reali e complesse: geometria, topologia e analisi armonica'', Project PRIN 2017 ``Real and Complex Manifolds: Topology, Geometry and holomorphic dynamics''
and by GNSAGA of INdAM}

\title{Cohomologies of complex manifolds with symplectic $(1,1)$-forms}

 \author{Adriano Tomassini and Xu Wang}
\date{\today}

\begin{abstract} Let $(X, J)$ be a complex manifold with a non-degenerated smooth $d$-closed $(1,1)$-form $\omega$. Then we have a natural double complex $\dbar+\dbar^\Lambda$, where $\dbar^\Lambda$ denotes the symplectic adjoint of the $\dbar$-operator. We study the Hard Lefschetz Condition on the Dolbeault cohomology groups of $X$ with respect to the symplectic form $\omega$. In \cite{TW}, we proved that such a condition is equivalent to a certain symplectic analogous of the $\partial\dbar$-Lemma, namely the $\dbar\, \dbar^\Lambda$-Lemma, which can be characterized in terms of Bott--Chern and Aeppli cohomologies associated to the above double complex. We obtain Nomizu type theorems for the Bott--Chern and Aeppli cohomologies and we show that the $\dbar\, \dbar^\Lambda$-Lemma is stable under small deformations of $\omega$, but not stable under small deformations of the complex structure. However, if we  further assume that $X$ satisfies the $\partial\dbar$-Lemma then  the $\dbar\, \dbar^\Lambda$-Lemma is stable.

\bigskip

\end{abstract}
\maketitle

\tableofcontents

\section{Introduction}\label{intro} 

It is known that the de Rham cohomology of a compact K\"ahler manifold satisfies two crucial properties: Hodge decomposition and Hard Lefschetz Condition, which do not hold for a general compact complex manifold. A natural question is to find a formal algebraic description of the above two properties. The first breakthrough is due to Fr\"olicher \cite{Fr} who proved that the first property is equivalent to that the Fr\"olicher spectral sequence degenerates at $E_1$, in particular, every compact surface satisfies the Hodge decomposition property.  In \cite{DGMS} Deligne--Griffiths--Morgan--Sullivan introduced the stronger notion of \emph{$\partial\dbar$-Lemma}, which turns out to be equivalent to the fact that the de Rham cohomology possesses both the Hodge decomposition property and the Hodge structure (see \cite[Proposition 5.12]{DGMS}); furhermore, they proved that every compact K\"ahler manifold satisfies the \emph{$\partial\dbar$-Lemma}. From \cite{Ma, Br, Merkulov, Yan}, we know that the Hard Lefschetz Condition on the de Rham cohomology  is essentially an integrability condition (the \emph{$dd^\Lambda$-Lemma}) on the associated differentiable Gerstenhaber-Batalin-Vilkovisky algebra. In particular, every compact K\"ahler manifold satisfies the $dd^\Lambda$-Lemma. For a general compact complex manifold, we know from the main theorem in \cite{AT13} that the $\partial\dbar$-Lemma is equivalent to a Fr\"olicher-type equality for Bott--Chern and Aeppli cohomologies. In  \cite[Def. 8.3]{TW}, we introduced the $\overline{\partial}\,\overline{\partial}^\Lambda$-{\em Lemma} and proved that it is equivalent to the Hard Lefschetz Condition on the \emph{Dolbeault cohomology group}. More precisely, let $(X, J)$ be a compact complex manifold with a symplectic $(1,1)$-form $\omega$. Denote by $\Lambda$ the \emph{symplectic adjoint} of  $L:=\omega \wedge$ (see \eqref{eq:sl2}), which satisfies $(\dbar^\Lambda)^2=(\dbar+\dbar^\Lambda)^2=0$. Then $(X, \omega, J)$ satisfies the $\overline{\partial}\,\overline{\partial}^\Lambda$-{\em Lemma} if every $\dbar$-closed, $\dbar^\Lambda$-closed, $\dbar+\dbar^\Lambda$-exact complex form is $\overline{\partial}\,\overline{\partial}^\Lambda$-exact. It has to be remarked that the $\overline{\partial}\,\overline{\partial}^\Lambda$-{\em Lemma} is a generalization of the $\overline{\partial}\partial^*$-Lemma on compact K\"ahler manifolds (see Sec. \ref{d-dbar-lemma}). Cohomologies associated to the  $\overline{\partial}\,\overline{\partial}^\Lambda$-{\em Lemma} are  complex symplectic Bott--Chern and Aeppli cohomologies (see Def. \ref{de:2.1}). In this paper, we shall show how to compute the above complex symplectic cohomologies and use them to study the deformation property of the $\dbar\,\dbar^\Lambda$-Lemma. Our first result is

\smallskip

\noindent {\bfseries Theorem A}{\bfseries.}\
{\itshape 
Let $(X, J)$ be a compact complex manifold with a symplectic $(1,1)$-form $\omega$. Write $H_\sharp(X):=\oplus H_\sharp^{p,q}(X)$ for $\sharp\in \{\dbar, \dbar^\Lambda, BC, A\}$ (see Def. \ref{de:2.1}). Then
\begin{enumerate}
 \item $H_{BC}(X)$ and $H_A(X)$ satisfy the Hard Lefschetz Condition with respect to $L$.
 \item With respect to an admissible Hermitian metric (see Definition \ref{de:2.5}),  both the space of $\triangle_{BC}$-harmonic forms $\mathcal{H}_{BC}(X)$ and the space of   $\triangle_{A}$-harmonic forms $\mathcal{H}_{A}(X)$ satisfy the Hard Lefschetz Condition with respect to $L$. But in general, $\mathcal{H}_{BC}(X)$ and $\mathcal{H}_{A}(X)$ are not an algebra with respect to the wedge product. In fact, the Kodaira--Thurston manifold in section \ref{se:kt} will give a counterexample.
 \item The the Kodaira--Thurston manifold in section \ref{se:kt} and the Iwasawa manifold in section \ref{se:Iwa} do not satisfy the $\overline{\partial}\,\overline{\partial}^\Lambda$-Lemma.
 \end{enumerate}
}
\smallskip

Theorem A (1) and (2) depend on a study of the harmonic representative of a complex symplectic cohomology class in Sec. \ref{harmonic} and \ref{se:kid} (see \cite{TY} for the real case). The main ingredient behind the proof of  Theorem A (2) is a certain Minkowski type K\"ahler identity associated to suitable Hermitian metric (see Def. \ref{de:2.5}). The proof of Theorem A (3) depends on an explicit computation of the associated cohomology group. The main idea is to to prove the following Nomizu type theorem (see \cite{Nomizu54, CF, CFK, CFGU, CFGU2, FRR, K, rollenske09b, rollenske10, Sa} for related results).

\smallskip

\noindent {\bfseries Theorem B}{\bfseries.}\
{\itshape 
Let $(X,J)$ be a compact complex manifold. Assume that its holomorphic cotangent bundle possesses a smooth global frame $\Psi=\{\xi^1, \cdots, \xi^n\}$. Let 
$$\omega=i\sum \omega_{j\bar k} \, \xi^j \wedge\overline{\xi^k}
$$ be a symplectic form on $X$ with constant coefficients $\omega_{j\bar k}$. Write $H_\sharp(X):=\oplus H_\sharp^{p,q}(X)$ for $\sharp\in \{\dbar, \dbar^\Lambda, BC, A\}$ (see Def. \ref{de:2.1}). Assume that  $
H_{\dbar}(X) $ is $\Psi$ reduced (see Definition \ref{de:psi-reduced}), then $H_{\sharp}(X)$ are also $\Psi$ reduced for $\sharp\in \{\dbar^\Lambda, BC, A\}$. 
In particular, if $\Psi$ is complex nilpotent (see \cite{RTW}) then $H_{\sharp}(X) $ are $\Psi$ reduced for $\sharp\in \{\dbar, \dbar^\Lambda, BC, A\}$.
}
\smallskip

The above theorem can be used to prove the following deformation property of the $\dbar\,\dbar^\Lambda$-Lemma and the Dolbeault formality (see \cite{TT} and Sec. \ref{se:defor} for the definition):

\medskip

\noindent {\bfseries Theorem C}{\bfseries.}\
{\itshape  Let $(X, J)$ be a compact complex manifold with a symplectic $(1,1)$-form $\omega$. Then 
\begin{enumerate}
\item $\overline{\partial}\,\overline{\partial}^\Lambda$-Lemma is stable with respect to $\omega$, more precisely, let $\{\omega_{t}\}_{|t|<1}$ be a smooth family of symplectic $(1,1)$-forms, if the $\overline{\partial}\,\overline{\partial}^\Lambda$-Lemma holds for $\omega_0$ then it holds for all $\omega_t$ with sufficiently small $|t|$;
\item  If $X$ satisfies the $\partial\dbar$-Lemma and $\dbar\,\dbar^\Lambda$-Lemma then so does any small deformation of $X$;
\item There exists a complex analytic family of three dimensional Nakamura manifolds such that 
the central fiber is geometrically Dolbeault formal and satisfies the $\dbar\,\dbar^\Lambda$-Lemma, but all the nearby fibers are not Dolbeault formal neither satisfy the $\dbar\,\dbar^\Lambda$-Lemma. In particular, the $\dbar\,\dbar^\Lambda$-Lemma is not a stable property under small deformations of the complex structure.
\end{enumerate}
} 

\medskip
The paper is organized as follows: in Section \ref{preliminaries} we start by recalling some facts on complex and symplectic geometry, introducing the complex symplectic cohomologies $H^{\bullet,\bullet}_{\dbar^\Lambda}(X)$, $H^{\bullet,\bullet}_{BC}(X)$, $H^{\bullet,\bullet}_{A}(X)$, and fixing some notation. In Section \ref{harmonic}, by using standard techniques, we prove a Hodge decomposition for the differential operators $\Box_{\dbar^\Lambda}$, $\triangle_{BC}$ and $\triangle_{A}$ naturally associated to the complex symplectic cohomologies. In Section \ref{d-dbar-lemma}, by applying a result in \cite[Theorem 3.4]{AT}, we give a characterization of the $\dbar\dbar^{\Lambda}$-Lemma in terms of the complex symplectic cohomologies (see Theorem \ref{d-dbar}). In Section \ref{se:kid} we prove a K\"ahler identity of Minkowsky type for complex manifolds endowed with a symplectic $(1,1)$-form admitting an admissible Hermitian metric. As a consequence, we obtain that the direct sum of the spaces of $\triangle_{BC}$-harmonic $(p,q)$-forms associated to an admissible Hermitian metric on a compact complex manifold satisfies the Hard Lefschetz Condition (see Theorem \ref{BC-HLC}). Sections \ref{se.AB} and \ref{se:defor} are devoted to the proof of Theorems A, B and C. 

\medskip

\textbf{Remark}:  Theorem C (1) suggests to study the following question:

\medskip

\textbf{Question 1}: \emph{Whether the Hard Lefschetz Condition on the Dolbeault cohomology group depends on the choice of symplectic $(1,1)$-forms or not? In particular, does the Hard Lefschetz Condition hold true with respect to any symplectic $(1,1)$-forms on a compact K\"ahler manifold? }

\medskip

\textbf{Remark}: It is known that the Hard Lefschetz Condition on the Dolbeault cohomology group does depend on the choice of symplectic (might not be $(1,1)$) form (see \cite[Theorem 1.3]{cho}), thus we believe that answer is "No"  to Question 1. But we could not find a counterexample.

\medskip

From (2) in the above theorem, one might also ask the following:

\medskip

\textbf{Question 2}: \emph{Let $(X, J)$ be a compact complex manifold with a symplectic $(1,1)$-form $\omega$. Does the  $\partial\dbar$-Lemma imply the $\dbar\,\dbar^\Lambda$-Lemma?}

\medskip

\textbf{Remark}: If the answer to Question 1 is "No" for some compact K\"ahler manifold then the answer to Question 2 is also "No", since every compact K\"ahler manifold satisfies the $\partial\dbar$-Lemma. 
\vskip.2truecm\indent
\noindent{\sl Acknowledgments.} The first author would like to thank the Department of Mathematical Sciences of Norwegian University of Science and Technology, Trondheim, for its warm hospitality.

\section{Preliminaries and notation}\label{preliminaries}
Let $(X,J)$ an $n$-dimensional compact complex manifold. Denote by $A^{p,q}(X)$ the space of $(p,q)$-forms on $X$. A $(1,1)$-{\em symplectic form} on $(X,J)$ is a symplectic form $\omega$ of type 
$(1,1)$ on $(X,J)$, that is $\omega$ is a symplectic form on $X$ which is $J$-invariant. Locally one may write $\omega=i \sum \omega_{j\bar k} d\xi^j \wedge d\bar \xi ^k$. Denote by $(\omega^{-1})^{\bar{r}j}$ the inverse matrix of $(\omega_{j\bar k} )$. Then for any given $\varphi$, $\psi\in A^{p,q}(X)$, one may define
$$
\omega^{-1}(\varphi,\psi):=\frac{1}{p!q!}\sum (\omega^{-1})^{\bar{r}_1j_1}\cdots 
(\omega^{-1})^{\bar{k}_1s_1}\cdots
(\omega^{-1})^{\bar{k}_qs_q}
\varphi_{j_1\cdots j_p\bar{k}_1\cdots  \bar{k}_q}\overline{\psi_{r_1\cdots r_p \bar{s}_1\cdots \bar{s}_q}}.
$$
Then 
the {\em symplectic star operator}
$ *_s:  \wedge ^p A^{p,q} \to  \wedge A^{n-q,n-p}(X)$ is defined by the following representation formula 
\begin{equation}
 i^{q-p} \,\varphi\wedge *_s\overline{\psi}=\omega^{-1}(\varphi, \psi) \frac{\omega^n}{n!}.
\end{equation}
Then $*_s$ is a real operator which can be extended $\mathcal{C}^\infty(X,\C)$-linearly to the space of complex differential forms $A^k(X)$ and $*_s^2=\hbox{\rm id}$. The $sl_2$-triple $\{L, \Lambda, B\}$ acting on the space of $(p,q)$-forms on  $(X,J,\omega)$ is defined by
\begin{equation}\label{eq:sl2}
L:=\omega\wedge, \ \ \Lambda:=*_s L*_s,  \ \ B:=[L, \Lambda].
\end{equation}
We define the {\em symplectic adjoint} $\dbar^\Lambda: A^{k}(X)\to A^{k-1}(X)$ of $\dbar$ as
\begin{equation}\label{eq:sl2-1}
\dbar^\Lambda:=(-1)^{k+1}*_s\dbar *_s.
\end{equation}
Then, as a consequence of \cite[Theorem A]{TW}, we have the following symplectic identity
\begin{equation}\label{eq:sl2-2}
\dbar^\Lambda=[\dbar,\Lambda].
\end{equation}
Setting as usual,
$$
H_{\dbar}^{p,q}(X):=\frac{\ker \dbar\cap A^{p,q}(X)}{{\rm Im} \,\dbar\cap A^{p,q}(X)},
$$
we recall the following two definitions
\begin{definition}[Complex-symplectic cohomologies]\label{de:2.1}
$$
H_{\dbar^\Lambda}^{p,q}(X):=\frac{\ker \dbar^\Lambda \cap A^{p,q}(X)}{{\rm Im} \,\dbar^\Lambda\cap A^{p,q}(X)},  \ \ \ \  
H_{BC}^{p,q}(X):=\frac{\ker \dbar\cap\ker \dbar^\Lambda \cap A^{p,q}(X)}{{\rm Im} \,\dbar\,\dbar^\Lambda\cap A^{p,q}(X)}, 
$$
and 
$$
H_A^{p,q}(X):=\frac{\ker {\dbar\,\dbar^\Lambda}\cap A^{p,q}(X)}{({\rm Im} \,{\dbar^\Lambda}+ {\rm Im} \,{\dbar})\cap A^{p,q}(X)} 
$$
\end{definition}
\begin{definition} (see \cite[Def. 8.3]{TW}) $(X,J,\omega)$ is said to satisfy 
 the $\overline{\partial}\,\overline{\partial}^\Lambda$-{\em Lemma} if
 $$
\ker\overline{\partial}\cap\ker\overline{\partial}^\Lambda\cap(\hbox{\rm Im}\,\overline{\partial}+\hbox{\rm Im}\,\overline{\partial}^\Lambda)=
\hbox{\rm Im}\,\overline{\partial}\overline{\partial}^\Lambda.
$$
\end{definition}
Finally, if $g$ is a Hermitian metric on $(X,J)$, with fundamental form $\omega_g$, then setting, for any given 
$\varphi,\psi\in A^{p,q}(X)$, 
$$
(\varphi,\psi)(x)=\frac{1}{p!q!}(g)^{\bar{r}_1j_1}\cdots 
(g)^{\bar{k}_1s_1}\cdots
(g)^{\bar{k}_qs_q}
\varphi_{j_1\cdots j_p\bar{k}_1\cdots  \bar{k}_q}\overline{\psi_{r_1\cdots r_p \bar{s}_1\cdots \bar{s}_q}},
$$
we denote by $\ll ,\gg$ the $L^2$-Hermitian product on $X$ defined as 
$$
\ll\varphi,\psi\gg=\int_X (\varphi,\psi)(x)\frac{\omega_g^n}{n!}
$$
\section{Complex symplectic cohomologies and Hodge Theory}\label{harmonic}
\subsection{Finiteness theorem}
Let $(X, J)$ be a compact complex manifold with a symplectic $(1,1)$-form $\omega$. For a given $J$-Hermitian metric $g$ on $X$, we will denote by $\omega_g$ the associated fundamental form. We start by giving the following 
\begin{definition} We set
\begin{equation}
\begin{array}{ll}
\Box_{\dbar}&:=\dbar\,\dbar^*+\dbar^*\,\dbar\\[10pt]
\Box_{\dbar^\Lambda}&:=\dbar^\Lambda(\dbar^\Lambda)^*+(\dbar^\Lambda)^*\dbar^\Lambda\\[10pt]
\triangle_{BC}&:=\dbar\,\dbar^\Lambda(\dbar^\Lambda)^*\dbar^*+
(\dbar^\Lambda)^*\dbar^*\dbar\,\dbar^\Lambda+
(\dbar^\Lambda)^*\dbar\,\dbar^*\dbar^\Lambda+
\dbar^*\dbar^\Lambda(\dbar^\Lambda)^*\dbar+
(\dbar^\Lambda)^*\dbar^\Lambda+\dbar^*\dbar
\\[10pt]
\triangle_{A}&:=\dbar\dbar^*+\dbar^\Lambda(\dbar^\Lambda)^*+\dbar^*(\dbar^\Lambda)^*
\dbar^\Lambda \overline{\partial}+ \dbar^\Lambda \dbar^* \overline{\partial} (\dbar^\Lambda)^* +\dbar^\Lambda\overline{\partial} \dbar^*(\dbar^\Lambda)^* +\overline{\partial} (\dbar^\Lambda)^* \dbar^\Lambda \dbar^*
\end{array}
\end{equation}
\end{definition}
We have the following 
\begin{lemma}\label{kernel}
Let $\psi\in A^{p,q}(X)$. Then,
\begin{enumerate}
\item[i)] $$
\psi\in\ker \Box_{\dbar}\iff 
\left\{\begin{array}{lll}
\dbar\psi&=0\\[5pt]
\dbar^*\psi&=0
 \end{array}
\right.
 $$
 \item[ii)] $$
\psi\in\ker \Box_{\dbar^\Lambda}\iff 
\left\{\begin{array}{lll}
\dbar^\Lambda\psi&=0\\[5pt]
(\dbar^\Lambda)^*\psi&=0
 \end{array}
\right.
 $$
 \item[iii)] $$
\psi\in\ker \triangle_{BC}\iff 
\left\{\begin{array}{lll}
\dbar\psi&=0\\[5pt]
\dbar^\Lambda\psi&=0\\[5pt]
(\dbar\,\dbar^\Lambda)^*\psi&=0
 \end{array}
\right.
 $$
 \item[iv)] $$
\psi\in\ker \triangle_{A}\iff 
\left\{\begin{array}{lll}
\dbar\,\dbar^\Lambda\psi&=0\\[5pt]
\dbar^*\psi&=0\\[5pt]
(\dbar^\Lambda)^*\psi&=0
 \end{array}
\right.
 $$
\end{enumerate}
\end{lemma}
\begin{proof}
i) It is well known from Hodge-Dolbeault theory. \newline
The proof of ii) is similar to the proof of i). \newline 
iii) Let $\psi\in A^{p,q}(X)$. Assume that 
$$
\dbar\psi=0,\qquad
\dbar^\Lambda\psi=0,\qquad
(\dbar\,\dbar^\Lambda)^*\psi=0.
$$
Then, clearly $\triangle_{BC}\psi=0$. 

Conversely, let $\triangle_{BC}\psi=0$. Then, by the definition of $\triangle_{BC}$, we easily get
$$
\begin{array}{lll}
0&=&\ll\triangle_{BC}\psi,\psi\gg\\[10pt]
&=&\vert(\dbar^\Lambda)^*\dbar^*\psi\vert^2+
\vert\dbar\,\dbar^\Lambda\psi\vert^2+
\vert\dbar^*(\dbar^\Lambda)^*\psi\vert^2+
\vert(\dbar^\Lambda)^*\dbar\psi\vert^2+
\vert\dbar^\Lambda\psi\vert^2+
\vert\dbar\psi\vert^2
\end{array}
$$
The last equation implies that 
$$
\dbar\psi=0,\qquad
\dbar^\Lambda\psi=0,\qquad
(\dbar\,\dbar^\Lambda)^*\psi=0.
$$
The proof of iv) is similar.
\end{proof}
\smallskip

The following theorem is essentially known. 
\begin{theorem}\label{th:finite}  Let $(X, J)$ be a compact $n$-dimensional complex manifold endowed with a symplectic $(1,1)$-form $\omega$. If
$\sharp\in 
\{\dbar, \dbar^\Lambda, BC, A\}$,  then the cohomology groups $H^{p,q}_{\sharp}(X)$ are  finite dimensional.
\end{theorem}
We shall give another proof of the above theorem using 
\emph{harmonic representatives}. 
The main idea is to use the following linear algebra lemma:

\begin{lemma}\label{star-symplectic} Let $(X, J)$ be a compact $n$-dimensional complex manifold with a symplectic $(1,1)$-form $\omega$. Fix a Hermitian metric $g$ with fundamental form $\omega_g$ on $X$. Denote by $*_s^g$ the associated symplectic star operators with respect to $\omega_g$. Then
\begin{itemize}
 \item[i)] $*_s^g*_s=*_s*_s^g$\vskip.2truecm
 \item[ii)] $(*_s)^*=*_s$
\end{itemize}
where $ (*_s)^*$ denotes the adjoint of $*_s$.
\end{lemma}
In order to prove Lemma \ref{star-symplectic}, we need the following (see e.g., \cite[Lemma 1.6]{Wang17})
\begin{lemma} (Guillemin Lemma) Let $(V,\omega)$ be a symplectic vector space. 
Assume that the  
$$
(V,\omega) = (V_1,\omega^1)\oplus (V_2,\omega^2),
$$
where $(V_i,\omega^i)$, $i=1,2$ are symplectic vector spaces. Then
$$
*_s(u\wedge v)=(-1)^{k_1k_2}*^1_s u\wedge *^2_s v,
$$
for every $u\in\Lambda^{k_1}V_1^*$, $v\in\Lambda^{k_2}V_2^*$.
\end{lemma}
\noindent {\em Proof of Lemma \ref{star-symplectic}}\  i) For the first formula, fix $x\in X$; then we can choose local coordinates near $x$ such that
$$
\omega_g(x)=\frac{i}{2}\sum dz^j \wedge d\bar z^j, \ \ 
\omega(x)=\frac{i}{2}\sum  \lambda_j dz^j \wedge d\bar z^j.
$$
Then by the Guillemin Lemma, it is enough to prove the one dimensional case: the proof of this fact is trivial. 

\medskip

ii) The second formula follows from the first and 
$$
*_s u \wedge v=u\wedge *_s v,
$$
where $u, v$ have the same degree.\hfill$\Box$
\smallskip

\begin{remark}
The symplectic star operator $*_s:A^{p,q}(X)\to A^{n-q,n-p}(X)$ induces an isomorphism 
$*_s:H^{p,q}_{\dbar}(X)\to H^{n-q,n-p}_{\dbar^\Lambda}(X)$, by setting, for any given $[u]_{\dbar}\in H^{p,q}_{\dbar}(X)$,
$$
*_s[u]_{\dbar} = [*_s u]_{\dbar^\Lambda}
$$
\end{remark}
 The above lemma implies the  $*_s$ isomorphism $H^{n-q,n-p}_{\dbar^\Lambda}(X)=*_s H^{p,q}_{\dbar}(X)$ is also true for the associated harmonic spaces $\mathcal{H}^{n-q,n-p}_{\dbar^\Lambda}$ and 
 $\mathcal{H}^{p,q}_{\dbar}$. More precisely, we have the following result:

\begin{proposition}\label{th2.5} We have $\Box_{\dbar^\Lambda}=*_s \Box_{\dbar}*_s$, in particular $\ker \Box_{\dbar^\Lambda}=*_s \ker\Box_{\dbar}$. Consequently, 
$$
*_s: \mathcal{H}^{p,q}_{\dbar}\to \mathcal{H}^{n-q,n-p}_{\dbar^\Lambda}
$$
is an isomorphism.
\end{proposition}

\begin{proof} By ii) of the above lemma, $\dbar^\Lambda =(-1)^{k+1}*_s\dbar*_s$ satisfies
$$
(\dbar^\Lambda)^*=(-1)^{k+1}(*_s)^{*}\dbar^*(*_s)^*=(-1)^{k+1}*_s\dbar^**_s,
$$
which gives $\Box_{\dbar^\Lambda}=*_s \Box_{\dbar}*_s$.
\end{proof}
As a consequence, we can state and prove the following Hodge decomposition, which implies Theorem \ref{th:finite}
\begin{theorem}
Let $(X, J)$ be a compact $n$-dimensional complex manifold with a symplectic $(1,1)$-form $\omega$. Denote by $g$ a Hermitian metric on $X$. Then, 
\begin{enumerate}
 \item[I)] $\Box_{\dbar}$, $\Box_{\dbar^\Lambda}$, $\triangle_{BC}$, $\triangle_{A}$ are elliptic self-adjoint differential operators and, consequently, their kernels are finite dimensional complex vector spaces.
 \item[II)] Denoting by $\mathcal{H}_{\dbar}^{p,q}$, $\mathcal{H}_{\dbar^\Lambda}^{p,q}$, $\mathcal{H}_{BC}^{p,q}$ and $\mathcal{H}_{A}^{p,q}$, respectively $\ker\Box_{\dbar}\big\vert_{A^{p,q}}$, \newline 
 $\ker\Box_{\dbar^\Lambda}\big\vert_{A^{p,q}}$, $\ker \triangle_{BC}\big\vert_{A^{p,q}}$, $\ker \triangle_{A}\big\vert_{A^{p,q}}$, then the following decompositions hold:
 \begin{eqnarray}
 A^{p,q}(X)&=&\mathcal{H}_{\dbar}^{p,q}\stackrel{\perp}{\oplus}\dbar A^{p,q-1}(X)
\stackrel{\perp}{\oplus}\dbar^* A^{p,q+1}(X)\\
  A^{p,q}(X)&=&\mathcal{H}_{\dbar^\Lambda}^{p,q}\stackrel{\perp}{\oplus}\dbar^\Lambda A^{p+1,q}(X)
\stackrel{\perp}{\oplus}(\dbar^\Lambda)^* A^{p-1,q}(X)\\
A^{p,q}(X)&=&\mathcal{H}_{BC}^{p,q}\stackrel{\perp}{\oplus}\dbar\dbar^\Lambda A^{p+1,q-1}(X)\stackrel{\perp}{\oplus}
\Big(\dbar^*A^{p,q+1}(X)+(\dbar^\Lambda)^* A^{p-1,q}(X)\Big)\\
A^{p,q}(X)&=&\mathcal{H}_{A}^{p,q}\stackrel{\perp}{\oplus}\Big(\dbar A^{p,q-1}(X)+\dbar^\Lambda A^{p+1,q}(X)\Big)\stackrel{\perp}{\oplus}
(\dbar\,\dbar^\Lambda)^*A^{p-1,q+1}(X),
\end{eqnarray}
where $\perp$ is taken with respect to the $L^2$-Hermitian product.
\item[III)] Given any pair $(p,q)$, we have the following isomorphisms
$$
H^{p,q}_{\dbar}(X)\simeq \mathcal{H}_{\dbar}^{p,q},\quad
H^{p,q}_{\dbar^\Lambda}(X)\simeq \mathcal{H}_{\dbar^\Lambda}^{p,q},\quad
H^{p,q}_{BC}(X)\simeq \mathcal{H}_{BC}^{p,q},\quad
H^{p,q}_{A}(X)\simeq \mathcal{H}_{A}^{p,q}
$$
\end{enumerate}
We will refer to II) as the Hodge decomposition.
\end{theorem}
\begin{proof}
 I) The ellipticity of $\Box_{\dbar}$ is well known. The above Proposition \ref{th2.5} implies that $\Box_{\dbar^\Lambda}$ is elliptic.\newline
Now we compute the principal symbol $\sigma (\triangle_{A})$ of the operator $\triangle_{BC}$.\smallskip

\underline{\em Claim} The principal symbol $\sigma$ of $\triangle_{BC}$ can be written as
$$
\sigma (\triangle_{BC})=\sigma(\Box_{\dbar^\Lambda})\sigma(\Box_{\dbar}),
$$
  
The main idea is to use the local computation in \cite[Proposition 3.3 and Theorem 3.5]{TY}. Since $\Lambda$ is a linear combination of contractions of vectors, we can write
$$
\Lambda^*=\sigma\wedge ,
$$ 
for some degree $(1,1)$-form $\sigma$, which implies that
$$
[\dbar, \Lambda^*]=(\dbar\sigma)\wedge
$$
is an order zero operator. Taking the adjoint, we get
$$
[\dbar^*, \Lambda]\sim 0,
$$
where $\sim$ means differ by a lower order operator. We claim that
\begin{equation}\label{eq:claim}
[\dbar^*, \dbar^\Lambda]\sim  0. 
\end{equation}
In fact,  $[\dbar^*, \Lambda]\sim 0$ gives
$$
[\dbar^*, \dbar^\Lambda]=[\dbar^*, [\dbar, \Lambda]]\sim [\Box_{\dbar}, \Lambda].
$$
Thus our claim follows from
$$
[\Box_{\dbar}, \Lambda]^*=[\sigma\wedge, \Box_{\dbar}] \sim 0,
$$
where $[\sigma\wedge, \Box_{\dbar}] \sim 0$ follows from the fact that the leading term of $\Box_{\dbar}$ is 
$$-\sum g^{\bar k j} \partial^2/\partial z_j \partial\bar z_k
$$
and
$$
[-\sum g^{\bar k j} \partial^2/\partial z_j \partial\bar z_k, \sigma\wedge]\sim 0.
$$
Notice that our claim implies that 
$$
\dbar\,\dbar^\Lambda(\dbar^\Lambda)^*
\dbar^* = -\dbar^\Lambda\dbar(\dbar^\Lambda)^*\dbar^*
\sim \dbar^\Lambda(\dbar^\Lambda)^* \dbar\, \dbar^*.
$$ 
A similar argument gives
$$
(\dbar^\Lambda)^* \dbar^*\, \dbar\,\dbar^\Lambda\sim  (\dbar^\Lambda)^*\dbar^\Lambda \dbar^* \,\overline{\partial}, \ \ 
(\dbar^\Lambda)^*\dbar\,\dbar^*\dbar^\Lambda \sim (\dbar^\Lambda)^*\dbar^\Lambda\dbar^*\,\dbar, \ \ 
\dbar^*\dbar^\Lambda(\dbar^\Lambda)^*\dbar \sim  \dbar^\Lambda(\dbar^\Lambda)^*\dbar\,\dbar^*.
$$
Thus we have
$$
\triangle_{BC} \sim \Box_{\dbar^\Lambda}\Box_{\dbar}.
$$
From the claim, it follows immediately that $\triangle_{BC}$ is elliptic. \newline
A similar argument also shows that $\triangle_{A}$ is elliptic. I) is proved.\medskip

\noindent II) The proof of II) is a direct consequence of the theory of elliptic operators on compact manifolds.
\medskip

\noindent III) The first isomorphism is well known. The second isomorphism follows immediately from Proposition \ref{th2.5}. We show that 
$$
H^{p,q}_{BC}(X)\simeq \mathcal{H}_{BC}^{p,q}.
$$
Let $\psi\in\mathcal{H}_{BC}^{p,q}$. Then the map 
$$
F:\mathcal{H}_{BC}^{p,q}\to H^{p,q}_{BC}(X),\qquad \psi\mapsto [\psi]
$$
is an isomorphism. Indeed, $F$ is $\C$-linear. Furthermore, $F$ is injective; $0=F(\psi)=[\psi]$ if and only if 
$\psi\in\hbox{\rm Im}\,\dbar\,\dbar^\Lambda$. Therefore $\psi\in\hbox{\rm Im}\,\dbar\,\dbar^\Lambda\cap\mathcal{H}_{BC}^{p,q}$ and consequently, by II), it follows that $\psi=0$.
\newline 
The map $F$ is also surjective: let $[\psi]\in H^{p,q}_{BC}(X)$. Then, by Hodge decomposition II)
$$
\psi = (\psi)_{H} + \dbar\,\dbar^\Lambda\eta +\dbar^*\mu +(\dbar^\Lambda)^*\nu.
$$
A direct computation shows that $\dbar^*\mu=0$ and $(\dbar^\Lambda)^*\nu=0$, since $\psi\in\ker\dbar\cap\ker\dbar^\Lambda$ and $(\psi)_{H}\in\mathcal{H}_{BC}^{p,q}$. Therefore, 
$[\psi]=[(\psi)_{H}]$ and $F$ is surjective, that is the map $F$ is an isomorphism. \newline
Similarly, $H^{p,q}_{A}(X)\simeq \mathcal{H}_{A}^{p,q}$. The proof is complete.
\end{proof}
\section{The $\dbar\,\dbar^\Lambda$-Lemma} \label{d-dbar-lemma}
Let $(X, J)$ be a compact complex manifold with a symplectic degree $(1,1)$-form $\omega$. 
As already remarked in Section \ref{intro}, the $\overline{\partial}\,\overline{\partial}^\Lambda$-{\em Lemma} is a generalization of 
the $\overline{\partial}\partial^*$-Lemma on a compact K\"ahler manifold. In fact, as a consequence of Hodge theory and K\"ahler identities, any compact K\"ahler manifold $M$ satisfies 
$$
\ker \dbar\,\cap\ker \partial^*\cap(\hbox{\rm Im}\,\overline{\partial}+\hbox{\rm Im}\,\partial^*)=
\hbox{\rm Im}\,\overline{\partial}\partial^*.
$$
Therefore, since $\partial^*=-i \dbar^\Lambda$, it follows immediately that \emph{every compact K\"ahler manifold satisfies the $\overline{\partial}\,\overline{\partial}^\Lambda$-Lemma}. Let 
$$
H^k_{\dbar}(X):=\bigoplus_{p+q=k}H_{\dbar}^{p,q}(X)
$$
and consider the map 
$L^k:H^{n-k}_{\dbar}(X)\to H^{n+k}_{\dbar}(X)$ induced by the Lefschetz operator $L:A^{p,q}(X)\to A^{p+1,q+1}(X)$ defined as $L\alpha=\omega\wedge\alpha$. Then, $\Big(\bigoplus_{k} H_{\dbar}^k, L,\dbar\Big)$ is a Lefschetz complex. Since $[\dbar,L]=0$, Theorem 3.5 in \cite{TW} implies:

\medskip

\begin{theorem}\label{th4.1} Let $(X,J)$ endowed with a symplectic form of degree $(1,1)$.
Then the following conditions are equivalent:
\begin{enumerate}
 \item[i)] $(X,J,\omega)$ satisfies the $\overline{\partial}\,\overline{\partial}^\Lambda$-Lemma
 \item[ii)] The Lefschetz complex 
 $$\Big(\bigoplus_{k\geq 0} H_{\dbar}^k(X), L,\dbar\Big)$$
 satisfies the Hard Lefschetz Condition.
\end{enumerate}
\end{theorem}
 
\medskip

Further more, in our case, all the above cohomologies are finite dimensional, thus we know that (see Lemma 5.15 in \cite{DGMS}, Lemma 5.41 in \cite{Ma} or  Lemma 2.4 in \cite{AT}) the $\dbar\dbar^\Lambda$-Lemma implies that all the above cohomologies have the same dimension. The converse is also true, a better version is the following fact proved in \cite[Theorem 3.4]{AT} : 
\begin{theorem}\label{d-dbar} Let $(X, J)$ be an $n$-dimensional compact complex manifold with a symplectic degree $(1,1)$-form $\omega$. Then the following inequalities hold
\begin{enumerate}
 \item[I)] $$
\dim H^{p,q}_{BC}(X)+ \dim H^{p,q}_{A}(X) \geq \dim H_{\dbar}^{p,q}(X)+ \dim H_{\dbar^\Lambda}^{p,q}(X).
$$ 
\item[II)] Furthermore, the equality in the above inequalities holds for all $p,q$ if and only if the 
$\dbar\,\dbar^\Lambda$-Lemma holds on $(X, J,\omega)$.
\end{enumerate}
\end{theorem}
\noindent{\em Proof.}\  I) Consider the following double complex
$$
(B^{\bullet, \bullet}(X), \dbar,\dbar^\Lambda), \ \ \ B^{-p,q}(X):=A^{p,q}(X).
$$
We know that $\dbar$ (resp. $\dbar^\Lambda$) is of type $(0,1)$ (resp. $(1,0)$). Thus Remark 3.5 in \cite{AT} gives
\begin{equation}\label{eq:5.1}
\dim H^{p,q}_{BC}(X)+ \dim H^{p,q}_{A}(X) \geq \dim H_{\dbar}^{p,q}(X)+ \dim H_{\dbar^\Lambda}^{p,q}(X).
\end{equation}
II) Now it suffices to prove the second part of the Theorem. Put
$$
T^k(X):= \bigoplus_{p+q=k} B^{p,q}(X)= \bigoplus_{q-p=k} A^{p,q}(X), \ \ \ D:=\dbar+\dbar^\Lambda,
$$
then one may define
$$
H_D^k(X):=\frac{\ker D\cap T^k(X)}{{\rm Im}\, D\cap T^k(X)}.
$$
By Theorem 2 in \cite{AT},  the followings are equivalent:
\begin{itemize}
\item[(1)] for every $-n\leq k\leq n$, we have
$$
\sum_{q-p=k} \Big(\dim H^{p,q}_{BC}(X)+ \dim H^{p,q}_{A}(X)\Big) = 2 \dim H_D^k(X);
$$ 
\item[(2)] the $\dbar\,\dbar^\Lambda$-Lemma holds.
\end{itemize}
In order to use the above result, we need the following 
\begin{lemma} We have 
$$
\dim H_D^k(X)= \sum_{q-p=k} \dim H_{\dbar}^{p,q}(X)=\sum_{q-p=k} \dim H_{\dbar^\Lambda}(X).
$$
\end{lemma}

\begin{proof} The second equality is trivial since $*_s$ gives the following isomorphism:
$$
H_{\dbar}^{p,q} \simeq H_{\dbar}^{n-q,n-p}.
$$
To prove the first equality, we use a similar argument as in the proof of \cite[Theorem 2.3]{C}. Notice that $\dbar^\Lambda=[\dbar, \Lambda] $ gives (by induction on $m$)
$$
\dbar\Lambda^m=\Lambda^m \dbar+m \Lambda^{m-1}\dbar^\Lambda,
$$
which gives
$$
\dbar(e^\Lambda \alpha)=\dbar\left(\sum \frac{\Lambda^k}{k!} \alpha\right)=\sum \frac{\Lambda^k}{k!} \dbar\alpha+ \sum \frac{\Lambda^{k-1}}{(k-1)!} \dbar^\Lambda\alpha=e^\Lambda  (\dbar+\dbar^\Lambda)\alpha,
$$
for every $\alpha\in T^k(X)$. Thus we have
$$
e^{-\Lambda}\dbar(e^\Lambda \alpha)=   (\dbar+\dbar^\Lambda)\alpha=D\alpha,
$$
hence the $D$-complex is equivalent to the $\dbar$-complex on $T^k(X)$ and the lemma follows.
\end{proof}

\begin{proof}[Proof of the second part of Theorem 2] Assume that the $\dbar\,\dbar^\Lambda$-Lemma holds, then we have that
$$
\dim H^{p,q}_{BC}(X)= \dim H^{p,q}_{A}(X) = \dim H_{\dbar}^{p,q}(X)= \dim H_{\dbar^\Lambda}^{p,q}(X),
$$
which gives 
\begin{equation}\label{eq:5.2}
\dim H^{p,q}_{BC}(X)+ \dim H^{p,q}_{A}(X) = \dim H_{\dbar}^{p,q}(X)+ \dim H_{\dbar^\Lambda}^{p,q}(X).
\end{equation}
On the other hand, \eqref{eq:5.2} and the above lemma together imply
$$
\sum_{q-p=k} \Big(\dim H^{p,q}_{BC}(X)+ \dim H^{p,q}_{A}(X)\Big) = 2 \dim H_D^k(X), \ \ \forall \ k,
$$
which is equivalent to that
the $\dbar\,\dbar^\Lambda$-Lemma holds (by \cite[Theorem 2]{AT}).
\end{proof}

\section{K\"ahler identities and admissible metrics}\label{se:kid}

\subsection{K\"ahler identitity of Minkowski type}
In this section we shall prove that if $\omega_g$ further satisfies the assumptions in the following lemma then a K\"ahler identity of Minkowski type holds.

\begin{lemma}\label{le2.6} Let $X$ be an $n$-dimensional complex manifold with Hermitian metric $\omega_g$. Let $\omega$ be a non-degenerate $(1,1)$-form on $X$. Let $\{L,\Lambda, B\}$ be the $sl_2$-triple associated to $\omega$. Let 
$$
\lambda_1 \leq \lambda_2 \leq \cdots \leq \lambda_n,
$$
be the eigenvalues of $\omega$ with respect to $\omega_g$. Assume that
$$
\lambda_j^2=1, \qquad \forall \ 1\leq j\leq n.
$$
Denote by $\Lambda^*$ the adjoint of $\Lambda$ with respect to $\omega_g$. Then
$$
\Lambda^*=L.
$$
\end{lemma}

\begin{proof}  Fix $x\in X$, then we can choose local coordinates near $x$ such that
$$
\omega_g(x)=i\sum dz^j \wedge d\bar z^j,
$$
and
$$
\omega(x)=i\sum  \lambda_j \,dz^j \wedge d\bar z^j.
$$
Let $\{V_1, \cdots, V_n\}$ be the dual frame of $\{dz^1, \cdots, dz^n\}$. Then we have
$$
\Lambda=i \sum \frac{1}{\lambda_j} \,(V_j\,\rfloor)(\overline{V_j}\,\rfloor).
$$
Thus 
$$
\Lambda^*=i\sum \frac{1}{\lambda_j}  dz^j \wedge d\bar z^j.
$$
Now we know that $\Lambda^*=\omega\wedge$ if and only if $\lambda_j^2=1$ for every $j$.
\end{proof}

We will introduce the following definition

\begin{definition}\label{de:2.5} A Hermitian metric $\omega_g$ is said to be admissible with respect to $\omega$ if all eigenvalues of $\omega$ with respect to $\omega_g$ lies in $\{1,-1\}$.
\end{definition}

\begin{theorem}[K\"ahler identity of Minkowski type] Let $(X,\omega, J)$ be a complex manifold with a symplectic $(1,1)$-form $\omega$. With respect to an admissible Hermitian metric $\omega_g$ we have
$$
(\dbar^\Lambda)^*=[L,\dbar^*], \ [(\dbar^\Lambda)^*, L]=0.
$$
We call them K\"ahler identities of Minkowski type.
\end{theorem}

\begin{proof} Follows from that the adjoint of $[\dbar, \Lambda]=\dbar^\Lambda$, $[\dbar^\Lambda, \Lambda]=0$ and $\Lambda^*=L$.
\end{proof}

\textbf{Remark 1}: In case $\omega$ is positive we know that $\omega_g$ is admissible with respect to $\omega$ if and only if $\omega=\omega_g$, in which case we have
$$
(\dbar^\Lambda)^*=-i\partial,
$$
thus the above theorem reduces to the usual K\"ahler identity.

\medskip

\textbf{Remark 2}: We know that each Bott-Chern type cohomology  $
 H^{p,q}_{BC}$
is isomorphic to 
$$
\mathcal H^{p,q}(BC):=\ker \dbar\cap\ker \dbar^\Lambda \cap \ker (\dbar\dbar^\Lambda)^* \cap A^{p,q}.
$$
Our K\"ahler identity of Minkowski type implies 

\begin{theorem}\label{BC-HLC} Let $(X,\omega, J)$ be a compact complex manifold with a symplectic $(1,1)$-form $\omega$. Let $\oplus \mathcal H^{p,q}(BC)$ be the above harmonic space associated to an arbitrary $\omega$ admissible metric, then $\{\oplus \mathcal H^{p,q}(BC), L:=\omega\wedge \cdot\}$ satisfies the Hard Lefschetz Condition.
\end{theorem}

\begin{proof} It is enough to show that for every $u\in  \mathcal H^{p,q}(BC)$, we have $Lu \in \mathcal H^{p,q}(BC)$. Notice that $\dbar u=0$ gives
$$
\dbar(Lu)=\omega \wedge \dbar u=0.
$$
Moreover, since $[\dbar, \Lambda]=\dbar^\Lambda$,  by the Jacobi identity, to show $\dbar^\Lambda(Lu)=0$, it is enough to prove
$$
[[L, \Lambda], \dbar]u=0,
$$
which follows directly from $\dbar u=0$ and
$$
[[L, \Lambda], \dbar]=\dbar.
$$
Now it suffices to show that $\dbar^* (\dbar^\Lambda)^* (Lu)=0$. By the K\"ahler identity of Minkowski type, we have $(\dbar^\Lambda)^*L=L (\dbar^\Lambda)^*$, which gives
$$
\dbar^* (\dbar^\Lambda)^* (Lu)=\dbar^* L (\dbar^\Lambda)^* u=[\dbar^*, L] (\dbar^\Lambda)^* u+ L \dbar^* (\dbar^\Lambda)^* u=[\dbar^*, L] (\dbar^\Lambda)^* u.
$$
Again our K\"ahler identity of Minkowski type gives
$$
[\dbar^*, L] (\dbar^\Lambda)^* u= - (\dbar^\Lambda)^*(\dbar^\Lambda)^*u=0.
$$
Thus the theorem follows.
\end{proof}


\subsection{A canonical way of choosing admissible  metric}

In general, admissible $J$-Hermitian metric is \emph{not unique}. In this section, we shall show that if the holomorphic cotangent bundle of $X$ is smoothly trivial then associated to a global frame, say 
$$
\Psi:=\{\xi^j\},
$$
there is a unique admissible $J$-Hermitian metric. In fact, assume that our symplectic form can be written as
$$
\omega=i \sum \omega_{j\bar k}\,  \xi^j \wedge \overline{\xi^k},
$$
where $\omega_{j\bar k}$ is a \emph{constant} Hermitian matrix with eigenvalues
$$
\lambda_1 \leq \cdots  \leq \lambda_s <0 < \lambda_{s+1}\leq \cdots \leq \lambda_n.
$$
Denote by $V_j$ the associated $\lambda_j$ eigenspace. Put
$$
V(-)=\oplus_{j=1}^s V_j, \ \ \  V(+)=\oplus_{j=s+1}^n V_j,
$$
Then one may define a $\omega$-admissible Hermitian metric $\omega_g$ such that 
$$
\omega_g(u, v)=0, \  \omega_g(u,u)=\omega(u,u), \ \omega_g (v,v)=-\omega(v,v).
$$
for every $u\in V(+)$ and $v\in V(-)$.

\begin{definition}\label{de:can-m} We call $\omega_g$ the canonical $\omega$-admissible metric associated to $\{\xi^j\}$.
\end{definition}

Denote by $A_{\Psi}^{p,q}$ the space of $(p,q)$-forms
$$
u=\sum u_{j_1\cdots j_p \overline{k_1}\cdots \overline{k_q}} \, \xi^{j_1}\wedge\cdots \wedge \xi^{j_p} \wedge \overline{\xi^{k_1}} \wedge\cdots\wedge \overline{\xi^{k_q}},
$$ 
where $u_{j_1\cdots j_p \overline{k_1}\cdots \overline{k_q}}$ are complex constants. Then one may define
$$
H_\sharp^{p,q}(\Psi), \ \ \  \sharp\in \{\dbar, \dbar^\Lambda, BC, A\},
$$
by replacing $A^{p,q}(X)$ with $A^{p,q}_{\Psi}$ in Def. \ref{de:2.1}. We shall introduce the following 

\begin{definition}\label{de:psi-reduced} We say that $H_\sharp^{p,q}(X)$, $\sharp\in \{\dbar, \dbar^\Lambda, BC, A\}$,  is $\Psi$ reduced if  the following isomorphism $H_\sharp^{p,q}(X)\simeq H^{p,q}_{\sharp}(\Psi)$ holds. 
\end{definition}

\section{Proofs of Theorems A and B}\label{se.AB}
In this section we will give the proofs of the first two results. We need some preliminary computations and results.
\subsection{Complex-symplectic cohomology on Kodaira-Thurston surface}\label{se:kt}

In this case, we consider the Kodaira-Thurston manifold $(X, J)$ with symplectic structure (see \cite[Section 5]{RTW}). Let $\R^4$ with coordinates $x^1,\ldots, x^4$ and consider the following product: given any $a=(a^1,\ldots,a^4),b=(b^1\ldots, b^4)\in\R^4$, set
$$
a*b=(a^1+b^1,a^2+b^2, a^3+a^1b^2+b^3,a^4+b^4).
$$
Then $(\R^4,*)$ is a Lie group and $\Gamma=\{(\gamma^1,\ldots,\gamma^4)\in\R^4\,\,\,\vert\,\,\, \gamma_j\in\Z, j=1,\ldots,4\}$ is a 
lattice in $(\R^4,*)$, so that 
$X=\Gamma\backslash\R^4$ is a $4$-dimensional compact manifold. Then, 
$$
e^1=dx^1,\quad e^2=dx^2,\quad e^3=dx^3-x^1dx^2,\quad e^4=dx^4,\quad
$$
are $\Gamma$-invariant $1$-forms on $\R^4$, and, consequently, they give rise to a gobal coframe on $X$. It is 
$$de^3=-e^1\wedge e^2,$$ 
the other differentials vanishing.
Set
$$
Je^1=-e^2,\quad Je^2=e^1,\quad Je^3=-e^4,\quad Je^4=e^3,\quad
$$
and
$$
\omega = e^{13}+e^{24},
$$
where $e^{ij}=e^i\wedge e^j$ and so on. 
Then $J$ is a complex structure on $X$, a global coframe of $(1,0)$-forms is given by
$$
\varphi^1=e^1+ie^2,\qquad \varphi^2=e^3+ie^4
$$
and $\omega$ is a $(1,1)$-symplectic structure on $X$. Explicitly,  
$$
\omega=\frac12(\varphi^1\wedge\overline{\varphi^2}
+\overline{\varphi^1}\wedge\varphi^2),
$$
and the $(1,0)$-coframe $\{\varphi^1, \varphi^2\}$ satisfies
$$
\begin{cases}
d \varphi^1=0, \\
d\varphi^2=-\frac i2 \varphi^1\wedge \overline{\varphi^1}.
\end{cases}
$$
Put
$$
\xi^1:=\varphi^1+i\varphi^2, \ \ \xi^2:=\varphi^1-i\varphi^2,
$$
then we have
$$
\omega=\frac{i}{4}  \left(\xi^1\wedge \overline{\xi^1}-\xi^2\wedge \overline{\xi^2}\right).
$$
Thus the canonical admissible $J$-Hermitian metric associated to $\{\xi^j\}$ is
$$
\omega_g=\frac{i}{4}  \left(\xi^1\wedge \overline{\xi^1}+\xi^2\wedge \overline{\xi^2}\right)= \frac{i}2 (\varphi^1\wedge \overline{\varphi^1}+  \varphi^2\wedge \overline{\varphi^2}).
$$
We will compute the following complex-symplectic harmonic space
$$
\mathcal H^{p,q}(BC):=\ker \dbar\cap\ker \dbar^\Lambda \cap \ker (\dbar\dbar^\Lambda)^* \cap A^{p,q}.
$$
By Theorem \ref{BC-HLC}, it is enough to compute the primitive harmonic space, say $P$, in $\oplus \mathcal H^{p,q}(BC)$. It is clear that
$$
\mathcal H^{p,q}(BC)\cap P =\ker \dbar\cap \ker (\dbar\dbar^\Lambda)^* \cap P.
$$
We know that
$$
\dbar^*=i(-1)^{p+q}*_s^g \partial *_s^g, \ (\dbar^\Lambda)^*=(-i)*_s^g*_s\partial*_s*_s^g,
$$
on $A^{p,q}$. Thus
$$
\ker (\dbar\dbar^\Lambda)^*=\ker (\partial
*_s\partial*_s*_s^g)=\ker(\partial*_s\partial*_s^g).
$$
Now we can use the main result in \cite{RTW} to prove the following theorem:

\begin{theorem}\label{th:KT-BC} All  harmonic forms in $\mathcal H^{p,q}_{BC}$ are $G$-invariant. More precisely, we have
\begin{equation*}
\begin{cases}
\mathcal H^{0,0}_{BC} = {\rm Span}_{\mathbb C}\,\langle 1\rangle, \\

\mathcal H^{1,0}_{BC}= {\rm Span}_{\mathbb C}\,\langle \varphi^1\rangle, \\

\mathcal H^{0,1}_{BC} = {\rm Span}_{\mathbb C}\,\langle \overline{\varphi^1},\ \overline{\varphi^2}\rangle, \\

\mathcal H^{2,0}_{BC} = {\rm Span}_{\mathbb C}\,\langle \varphi^1 \wedge \varphi^2 \rangle,\\

\mathcal H^{1,1}_{BC} = {\rm Span}_{\mathbb C}\,\langle \overline{\varphi^1} \wedge \varphi^2,\ \varphi^1\wedge\overline{\varphi^2}, \varphi^1\wedge\overline{\varphi^1}\rangle, \\

\mathcal H^{0,2}_{BC} = {\rm Span}_{\mathbb C}\,\langle \overline{\varphi^1} \wedge \overline{\varphi^2}\rangle, \\

\mathcal H^{2,1}_{BC} = {\rm Span}_{\mathbb C}\,\langle \varphi^1 \wedge \varphi^2 \wedge \overline{\varphi^1}\rangle,\\

\mathcal H^{1,2}_{BC} = {\rm Span}_{\mathbb C}\,\langle  \varphi^2\wedge \overline{\varphi^1} \wedge \overline{\varphi^2}
,  \varphi^1\wedge \overline{\varphi^1} \wedge \overline{\varphi^2}\rangle,\\

\mathcal H^{2,2}_{BC} = {\rm Span}_{\mathbb C}\,\langle \varphi^1 \wedge\varphi^2\wedge\overline{\varphi^1} \wedge  \overline{\varphi^2} \rangle.
\end{cases}
\end{equation*}
\end{theorem}

\begin{proof}  $\mathcal H^{0,0}_{BC} = {\rm Span}_{\mathbb C}\,\langle 1\rangle$ is trivial.

\medskip

\emph{Degree $(1,0)$ case}: Notice that, by bidegree reasons, $\mathcal H^{1,0}_{BC}\subset\mathcal H^{1,0}_{\dbar}={\rm Span}_{\mathbb C}\,\langle \varphi^1\rangle$. By a direct computation, 
$\varphi^1\in\mathcal H^{1,0}_{BC}$, so that $\mathcal H^{1,0}_{BC}={\rm Span}_{\mathbb C}\,\langle \varphi^1\rangle$.

\medskip

\emph{Degree $(0,1)$ case}:  Let $u\in A^{0,1}(X)$. Then, by degree reasons,
$$
u\in\mathcal H^{0,1}_{BC};
u\in\mathcal H^{0,1}_{BC}
$$
if and only if 
$$
\dbar u=0,\qquad (\dbar^\Lambda)^*\dbar^* u=0.
$$
Notice that
$$
(\dbar^\Lambda)^*\dbar^* u=0 \iff 
\del*_s\del *_gu=0 \iff
*_s\del *_gu\,\, \hbox{\rm is a constant} \iff
\dbar^* u\,\, \hbox{\rm is a constant},
$$
which is equivalent to $\dbar\dbar^* u=0$. Thus we have $ \mathcal H^{0,1}_{BC}=\mathcal H^{0,1}_{BC}$.

\medskip

\emph{Degree $(2,0)$ case}:  Follows from 
$\mathcal H^{2,0}_{BC}=\mathcal H^{2,0}_{\dbar}$.

\medskip

\emph{Degree $(1,1)$ case}:  Let 
$u\in \mathcal H^{1,1}_{BC}$. We can write
\begin{equation}\label{11}
u=u_0+\dbar v, \qquad u_0\in \mathcal H^{1,1}(\dbar).
\end{equation}
We have:
$$
\mathcal H^{1,1}_{\dbar}={\rm Span}_{\mathbb C}\,\langle \varphi^1\wedge\overline{\varphi^2},\varphi^2\wedge\overline{\varphi^1}\rangle
$$
Then, it is easy to check that 
$$
\mathcal H^{1,1}_{\dbar} \subset \mathcal H^{1,1}_{BC}, 
$$
\underline{\em Claim}\hskip1.5truecm  $\dbar v\in P\cap \mathcal H^{1,1}_{BC}$.\smallskip\noindent

First of all, $\dbar v\in P$. Indeed, 
$$
\dbar v\in P\iff \Lambda\dbar v=0\iff -[\dbar,\Lambda]v=0\iff \dbar^\Lambda v=0
$$
Furthermore, by \eqref{11}, we get
$$
0=\dbar^\Lambda u=\dbar^\Lambda u_0+\dbar^\Lambda\dbar v=\dbar^\Lambda\dbar v=-\dbar\,\dbar^\Lambda v
$$
that is $\dbar^\Lambda v$ is a constant and, consequently, 
$$
\dbar^\Lambda u\in\hbox{\rm Im}\,\dbar^\Lambda\cap \mathcal{H}^{0,0}_{\dbar^\Lambda},
$$
which implies $\dbar^\Lambda v=0$, i.e., 

$\dbar v\in P\cap \mathcal H^{1,1}(BC)$, i.e. $\dbar^\Lambda\in P$. Moreover, by degree reasons, $\dbar^* v=0$, so that
$$
\dbar v\in\ker\dbar\cap\ker\,\dbar^\Lambda\cap\ker\,(\dbar^\Lambda)^*\dbar^*=\mathcal{H}_{BC}^{1,1}
$$
Now we can write
$$
v=v_0+\dbar^\Lambda f, \qquad v_0\in \mathcal H^{1,0}_{\dbar^\Lambda}.
$$
Since
$$
\mathcal H^{1,0}_{\dbar^\Lambda}=*_s \mathcal H^{2,1}_{\dbar}= {\rm Span}_{\mathbb C}\,\langle \varphi^1, \varphi^2\rangle,
$$
and
$$
\dbar \mathcal H^{1,0}_{\dbar^\Lambda}={\rm Span}_{\mathbb C}\,\langle \varphi^1\wedge\overline{ \varphi^1}\rangle \subset P\cap  \mathcal H^{1,1}_{BC},
$$
we have
$$
\dbar\,\dbar^\Lambda f \in P\cap  \mathcal H^{1,1}_{BC}. 
$$
Thus $\dbar\dbar^\Lambda f =0$ and our formula follows, that is 
$$
\mathcal H^{1,1}_{BC}={\rm Span}_{\mathbb C}\,\langle \varphi^1\wedge\overline{\varphi^1},\varphi^1\wedge\overline{\varphi^2},\varphi^2\wedge\overline{\varphi^1}\rangle
$$
\medskip

\emph{Degree $(0,2)$ case}:  Notice that $u\in
\mathcal H^{0,2}_{BC}$ if and only if
$$
\partial *_s \partial u=0.
$$
Taking the conjugate of the last equation, we obtain
$$
\dbar *_s \dbar\overline{u}=0.
$$
Thus, we have
$$
 *_s \dbar\overline{u}\in \mathcal H^{1,0}_{\dbar}={\rm Span}_{\mathbb C}\,\langle \varphi^1\rangle,
$$
which gives
$$
\dbar\overline{u} \in {\rm Span}_{\mathbb C}\,\langle \varphi^1\wedge\varphi^2 \wedge\overline{\varphi^1} \rangle.
$$
Thus $\dbar\overline{u}=0$, i.e.,
$$
\bar u\in \mathcal H^{2,0}_{\dbar}={\rm Span}_{\mathbb C}\,\langle \varphi^1 \wedge \varphi^2 \rangle.
$$
Therefore, 
$$
\mathcal{H}^{0,2}_{BC}={\rm Span}_{\mathbb C}\,\langle \overline{\varphi^1} \wedge\overline{\varphi^2} \rangle.
$$

\medskip

\emph{The remaining cases follow from the Hard Lefschetz property}.
\end{proof}

\subsection{Complex-symplectic Iwasawa manifold}\label{se:Iwa}

Consider the following three dimensional complex Heisenberg group 
\begin{equation}\label{eq:heisenberg}
\mathbb H(3, \mathbb C):=
\left\lbrace \begin{bmatrix}
1 & z_1 & z_3 \\
0 & 1 & z_2 \\
0 & 0 & 1
\end{bmatrix} 
: z_j\in \mathbb C,  \  j=1,2,3
\right\rbrace
\end{equation}
with the product induced by matrix multiplication. Identify an element in $\mathbb H(3, \mathbb C)$ by a vector, then one may write the product as
$$
(a,b, c) \cdot(z_1, z_2, z_3)=(z_1+a, z_2+c, z_3+az_2+b),
$$
from which we know that
$$
\psi^1:= d\bar z_1, \ \ \psi^2:=dz_2, \ \ \psi^3:=dz_3-z_1dz_2
$$
are left invariant one forms satisfying
\begin{equation}\label{eq:Iwa-eq}
\begin{cases}
d\psi^1=0 \\
d\psi^2=0 \\
d\psi^3=-\overline{\psi^1}\wedge \psi^2.
\end{cases}
\end{equation}
Let $J$ be the almost complex structure on $\mathbb H(3, \mathbb C)$ with global type $(1,0)$ frame $\{\psi^1, \psi^2,\psi^3\}$. Then the above equation implies that $J$ is integrable. Fix a lattice, say
$$
\Gamma:=\{(a,b, c)\in \mathbb H(3,\mathbb C): a, b, c\in \mathbb Z[i]\},
$$
in $\mathbb H(3,\mathbb C)$ and consider the left quotient
$$
X:=\Gamma \backslash \mathbb H(3,\mathbb C).
$$
Since $\{\psi^1, \psi^2,\psi^3\}$ is well defined on $X$, we know that $J$ induces a complex structure (still denoted by $J$) on $X$. Consider 
$$
\omega:=i\, \psi^2\wedge \overline{\psi^2}+ \psi^1\wedge \overline{\psi^3} - \psi^3\wedge \overline{\psi^1},
$$
we know that
$$
\omega^3= 6 i \, \psi^2\wedge \overline{\psi^2} \wedge \psi^1\wedge \overline{\psi^1} \wedge \psi^3\wedge \overline{\psi^3}\neq 0
$$
and
$$
d\omega=0.
$$
Thus $\omega$ is a type $(1,1)$ symplectic form on $X$. The canonical admissible $J$-Hermitian metric is
$$
\omega_g:=i\, \psi^2\wedge \overline{\psi^2}+ i\,\psi^1\wedge \overline{\psi^1} + i
\,\psi^3\wedge \overline{\psi^3}.
$$
Since $\Psi:=\{\psi^1, \psi^2,\psi^3\}$ is complex nilpotent,  theorem $B$ gives
$$
H_{BC}(X)\simeq H_{BC}(\Psi) , \ \  \ \ H_{\dbar}(X)\simeq H_{\dbar}(\Psi).
$$

\begin{theorem}\label{th:iwasawa} The above Iwasawa manifold does not satisfy the $\overline{\partial}\,\overline{\partial}^\Lambda$-Lemma.
\end{theorem}

\begin{proof}  It suffices to show that $\overline{\psi^1}$ is $\Box_{\dbar}$-harmonic but $
\omega^2\wedge \overline{\psi^1} $ is $\dbar$-exact. To show that $\overline{\psi^1}$ is $\Box_{\dbar}$-harmonic,  it is enough to verify that
$$
\dbar \,\overline{\psi^1} =0, \ \  \dbar*_s^g  \overline{\psi^1}=0.
$$
The first identity follows directly from \eqref{eq:Iwa-eq}. For the second identity, notice that up to a constant $*_s^g  \overline{\psi^1}$ is equal to $\omega_g^2 \wedge \overline{\psi^1}$. Again, by \eqref{eq:Iwa-eq}, we know that
$$
\omega_g^2 \wedge \overline{\psi^1} =-2 \psi^2\wedge \overline{\psi^2}\wedge \psi^3\wedge \overline{\psi^3} \wedge \overline{\psi^1}
$$
is $\dbar$-closed, which implies that
$$
\dbar*_s^g  \overline{\psi^1}=0.
$$
Hence $\overline{\psi^1}$ is $\Box_{\dbar}$-harmonic. Moreover, we have 
$$
\omega^2\wedge \overline{\psi^1}=2i \, \psi^2\wedge \overline{\psi^2}\wedge \psi^1\wedge \overline{\psi^3} \wedge \overline{\psi^1}=\dbar (2i \, \overline{\psi^2}\wedge \psi^1\wedge \overline{\psi^3} \wedge \psi^3),
$$
thus $\omega^2\wedge \overline{\psi^1}$ is $\dbar$-exact, from which we know that the $H_{\dbar}$ does not satisfy the Hard Lefschetz Condition. Thus our theorem follows from Theorem \ref{th4.1}.
\end{proof}

\subsection{Proof of Theorem A}
(1) Follows from Theorem 3.3 in \cite{TW} (see \cite{TY} for the real case).

\medskip

(2) The first part follows from Theorem \ref{BC-HLC}. For the second part, by the previous computations collected in Theorem \ref{th:KT-BC}, we immediately obtain that
$$
(\varphi^1\wedge \varphi^2)\wedge \overline{\varphi^2}
$$
is not $\triangle_{BC}$-harmonic, but both $\varphi^1\wedge \varphi^2$ and $\overline{\varphi^2}$ 
are $\triangle_{BC}$-harmonic. Consequently, $\mathcal{H}_{BC}(X)$ is not an algebra. 

\medskip

(3)  By \cite[Theorem B (4)]{TW}, we know that the Kodaira--Thurston manifold does not satisfy the $\overline{\partial}\,\overline{\partial}^\Lambda$-Lemma. The Iwasawa case follows from Theorem \ref{th:iwasawa}.
\newline
The Proof of Theorem A is complete.\hfill$\Box$
\subsection{Proof of Theorem B}
Now it suffices to prove Theorem B. Assume that
$$
H_{\dbar}(X)\simeq H_{\dbar}(\Psi).
$$ 
Since $w\in A_{\Psi}^{1,1}$, we know that 
$$
*_s (A_\Psi)=A_{\Psi},
$$
which gives
$$
H_{\dbar^\Lambda}(X)\simeq *_s H_{\dbar}(X)\simeq *_s H_{\dbar}(\Psi) \simeq H_{\dbar^\Lambda}(\Psi).
$$
Moreover, there is a natural map  from $A$ to $A_\Psi$ defined by
$$
\mu:u\mapsto\sum \left(\int_X u_{j_1\cdots j_p \overline{k_1}\cdots \overline{k_q}} \, \frac{\omega^n}{\int_X \omega^n}\right) \xi^{j_1}\wedge\cdots \wedge \xi^{j_p} \wedge \overline{\xi^{k_1}} \wedge\cdots\wedge \overline{\xi^{k_q}},
$$
for
$$
u=\sum u_{j_1\cdots j_p \overline{k_1}\cdots \overline{k_q}} \, \xi^{j_1}\wedge\cdots \wedge \xi^{j_p} \wedge \overline{\xi^{k_1}} \wedge\cdots\wedge \overline{\xi^{k_q}} \in A^{p,q}(X).
$$
Denoting by $\iota$ the natural mapping 
$$
\iota: A^{\bullet,\bullet}(\psi)\hookrightarrow A^{\bullet,\bullet}(X),
$$
notice that $\mu$ satisfies
$$
(\mu\circ\iota) (u)=u, \ \ \forall \ u\in A_{\Psi}.
$$
Thus Corollary 1.3 in \cite{AK19} implies that $H_{\sharp}(X) \simeq H_{\sharp}(\Psi)$ also for all $\sharp\in \{BC, A\}$. Moreover, in case $\Psi$ is complex nilpotent, the main theorem in \cite{RTW} implies $H_{\dbar}(X)\simeq H_{\dbar}(\Psi)$. Thus the above argument gives
$H_{\sharp}(X) \simeq H_{\sharp}(\Psi)$ for all $\sharp\in \{\dbar, \dbar^\Lambda, BC, A\}$. The proof is complete.

\section{Deformations of Nakamura manifolds}\label{se:defor}
This section is devoted to the proof of Theorem C. First of all,  
we need to recall some definitions and facts from {\em Dolbeault formality} on complex manifolds. By definition, a complex manifold $(X,J)$ is said to be 
{\em Dolbeault formal} if the bi-differential, bi-graded algebra (shortly {\em bba}) $( A^{\bullet,\bullet}(X),\del,\dbar)$ is equivalent (in the category of bba) to a bba $(B,\del_B,0)$, which means that there exists a family of bba 
$\{(C_l,\del_l,\dbar_l)\}_{l\in \{0,1,\dots,2n+2\}}$ such that $(C_0,\del_0,\dbar_0)=(A^{\bullet,\bullet}(X),\del,\dbar)$, 
$(C_{2n+2},\del_{2n+2},\delbar_{2n+2})=(B,\del_B,0)$ and a family of bba-morphisms
$$ \xymatrix{
   & \left(C_{2j+1},\del_{2j+1},\delbar_{2j+1} \right) \ar[ld]_{f_j} \ar[rd]^{g_j} & \\
  \left(C_{2j},\del_{2j},\delbar_{2j}\right) & & \left(C_{2j+2},\del_{2j+2},\delbar_{2j+2} \right)
} $$
for $l\in\{0,1,\dots,n\}$, such that the morphisms induced in cohomology are bba-isomorphisms. A complex manifold $(X,J)$ is said to be 
{\em geometrically Dolbeault formal} if there is a Hermitian metric $g$ such that the harmonic space of the Dolbeault cohomology is an algebra with respect to the wedge product. In particular, any complex manifold geometrically Dolbeault formal is Dolbeault formal. We now recall shortly the construction of {\em Dolbeault-Massey triple products} on a complex manifold, which provide an obstruction to Dolbeault formality.
Let 
$$
\mathfrak{a}=[\alpha]\in H_{\dbar}^{p,q}(X),\quad \mathfrak{b}=[\beta]\in H_{\dbar}^{r,s}(X),\quad 
\mathfrak{c}=[\gamma]\in H_{\dbar}^{u,v}(X)
$$
such that
$$
\mathfrak{a}\cdot\mathfrak{b}=0\in H_{\dbar}^{p+r,q+s}(X),\quad \mathfrak{b}\cdot\mathfrak{c}=0\in H_{\dbar}^{r+u,s+v}(X).
$$
Then there exist $f\in \Lambda^{p+r,q+s-1}X$ and $g\in \Lambda^{r+u,s+v-1}X$ satisfying
$$
\alpha\wedge\beta=\delbar f,\quad \beta\wedge\gamma=\delbar g.
$$
The {\em Dolbeault-Massey triple product} of the cohomology classes $\mathfrak{a},\,\mathfrak{b},\,\mathfrak{c}$ is defined as 
$$
\begin{array}{ccl}
\left\langle\mathfrak{a},\mathfrak{b},\mathfrak{c}\right\rangle & := & [f\wedge\gamma +(-1)^{p+q+1}
\alpha\wedge g] \\[8pt]
 & \in & \displaystyle \frac{H_{\delbar}^{p+r+u,q+s+v-1}(X)}{H_{\delbar}^{p+r,q+s-1}(X)\cdot H_{\delbar}^{u,v}(X)+
 H_{\delbar}^{p,q}(X)\cdot 
 H_{\delbar}^{r+u,s+v-1}(X)}.\\[8pt]
\end{array}
$$
Finally, if $(X,J)$ is Dolbeault formal, in particular geometrically formal, then all the Dolbeault-Massey triple products vanish.

\subsection{Complex and symplectic structures on Nakamura manifolds}
We start by recalling the construction and the cohomology properties of the holomorphically parallelizable Nakamura manifold (see \cite[p.90]{N}).
On $\C^3$ with coordinates $(z_1,z_2,z_3)$ 
consider the following product $*$
$$
(w_1,w_2,w_3)*(z_1,z_2,z_3)=(w_1+z_1,e^{w_1}z_2+w_2,e^{-w_1}z_3+w_3).
$$
Then $G=(\C^3,*)$ is a solvable Lie group, which is the semidirect product $\C\ltimes\C^2$, admitting a uniform 
discrete subgroup $\Gamma=\Gamma'\ltimes\Gamma''$, where $\Gamma'\subset\C$ is given by $\Gamma'=\lambda\Z\oplus i2\pi\Z$ and  
$\Gamma''$ is a lattice in $\C^2$; thus $N:=\Gamma\backslash \C^3$ is a compact complex $3$-dimensional 
manifold, endowed with the complex structure $J_N$ induced by the standard complex structure on $\mathbb C^3$. It turns out that $h^{0,1}(N)=3$. It is immediate to check that 
$$
\varphi^1=dz_1,\qquad \varphi^2=e^{-z_1}dz_2,\qquad \varphi^3=e^{z_1}dz_3
$$
are $G$-invariant holomorphic $1$-forms on $\C^3$, so that they induce holomorphic $1$-forms on $N$, namely 
$\{\varphi^1,\varphi^2,\varphi^3\}$ is a global holomorphic co-frame on $N$ and the complex manifold $N$ is holomorphically parallelizable. 
We have
$$
d\varphi^1=0,\qquad d\varphi^2=-\varphi^1\wedge\varphi^2,\qquad d\varphi^3=\varphi^1\wedge\varphi^3. 
$$
By the construction of $N$, it follows that $e^{\frac{z_1-\overline{z}_1}{2}}$ is a well-defined complex-valued smooth function on $N$. 
Let 
$$
\omega_N=\frac{i}{2}\varphi^1\wedge\overline{\varphi^1}+\frac12 e^{-z_1+\overline{z}_1}\,\overline{\varphi^2}\wedge\varphi^3
+\frac12 e^{z_1-\overline{z}_1}\,\varphi^2\wedge\overline{\varphi^3}.
$$
Then  $$
\overline{\omega_N}=\omega_N,\qquad \omega_N^3=-\frac{3}{4} (idz_1\wedge d\bar z_1) \wedge (idz_2\wedge d\bar z_3)\wedge(idz_3\wedge d\bar z_3)<0,
$$ 
and explicitly,
$$
\omega_N=\frac{i}{2} dz_1\wedge d\overline{z}_1 + \frac12 d\overline{z}_2\wedge dz_3+
\frac12 d\overline{z}_3\wedge dz_2\,,
$$
so that $d\omega_N=0$ and the complex structure $J_N$ on $N$ is $\omega$-symmetric. Then, see \cite[Sec. 8.4]{TW}, $(N,J_N,\omega_N)$ satisifies 
the $\overline{\partial}\,\overline{\partial}^\Lambda$-Lemma. By \cite{K}, the Dolbeault cohomology $N$ can be computed by taking the finite 
dimensional subcomplex $(C_\Gamma,\delbar)\hookrightarrow (A^{\bullet,\bullet}(N),\delbar)$ given by
$$
C_\Gamma =\Lambda^{\bullet,\bullet}\left(\Span_\C\left\langle dz_1,e^{-z_1}dz_2,e^{z_1}dz_3 \right\rangle
\oplus \Span_\C\left\langle d\overline{z}_1, e^{-z_1}d\overline{z}_2,e^{z_1}d\overline{z}_3\right\rangle\right).
$$

Let $g$ be the Hermitian metric on $N$ defined by
$$
g=\sum_{j=1}^3\varphi^j\otimes\overline{\varphi^j}
$$
and denote by $\Box^g_{\delbar}$  the Dolbeault Laplacian associated to $g$.
Then, it turns out that 
$$
H^{\bullet,\bullet}_{\delbar}(N)\simeq \ker \Box^g_{\delbar}=C_\Gamma,
$$
and that $N$ is geometrically Dolbeault formal (i.e. the harmonic space of the Dolbeault cohomology is an algebra with respect to the wedge product).
Summing up, $(N,J_N,\omega_N)$ is a compact $3$-dimensional geometrically Dolbeault formal complex manifold satisfying the 
$\overline{\partial}\,\overline{\partial}^\Lambda$-Lemma. 
\subsection{Complex deformations of Nakamura manifolds which do not satisfy the $\dbar\dbar^\Lambda$-Lemma}
We will construct a $1$-parameter complex deformation $N_t=(N,J_t)$ of $N=(N,J_N)$, 
admitting a $J_t$-symmetric symplectic structure $\omega_t$, such that 
$N_t$ is not Dolbeault formal (see Lemma \ref{massey} and \cite{TT} for the Definition) and $(N,J_t,\omega_t)$ does not satisfy the $\overline{\partial_t}\,\overline{\partial_t}^{\Lambda}$-Lemma, for $t\neq 0$. \newline
Let $\{\zeta_1,\zeta_2,\zeta_3\}$ be the holomorphic global frame on $N$, dual to $\{\varphi^1,\varphi^2,\varphi^3\}$. Then
$$
\zeta_1=\frac{\partial}{\partial z_1},\qquad \zeta_2=e^{z_1}\frac{\partial}{\partial z_2},\qquad \zeta_3=e^{-z_1}\frac{\partial}{\partial z_3}.
$$
\begin{lemma}\label{deformed-complex-structure}
Let $\varphi_t=te^{\overline{z}_1-z_1}\overline{\varphi^2}\otimes\zeta_3 \in A^{0,1}(N,T^{1,0}N)$, $t\in \C, \vert t\vert<\varepsilon$. Then
$$\delbar\varphi_t +\frac12 [[\varphi_t,\varphi_t]]=0.$$
\end{lemma}
\begin{proof}
By definition, $\varphi_t=t e^{-2z_1}d\overline{z}_2\otimes\frac{\partial}{\partial z_3}$. Therefore
$$
\delbar\varphi_t =\delbar (t e^{-2z_1}d\overline{z}_2)\otimes\frac{\partial}{\partial z_3} = 0
$$
and
$$
[[\varphi_t,\varphi_t]]=0.
$$
\end{proof}
According to Lemma \ref{deformed-complex-structure}, $\varphi_t$ determines an integrable complex structure $J_t$, for 
$t\in \mathbb{B}(0,\varepsilon)$. Denote by $N_t=(N,J_t)$.
\begin{lemma}
The following complex $1$-differential forms 
\begin{equation}\label{t-complex-forms}
\begin{array}{lll}
\Phi_1^{1,0}(t):=dz_1, &  \Phi_2^{1,0}(t):=e^{-z_1}dz_2, & \Phi_3^{1,0}(t):=e^{z_1}dz_3-te^{-z_1}d\overline{z}_2,\\[10pt]
\Phi_1^{0,1}(t):=d\overline{z}_1, &  \Phi_2^{0,1}(t):=e^{-z_1}d\overline{z}_2, & \Phi_3^{0,1}(t):=e^{z_1}d\overline{z}_3-
\overline{t}e^{z_1-2\overline{z_1}}dz_2,
\end{array}
\end{equation}
define a global coframe of $(1,0)$-forms, $(0,1)$-forms respectively on $N_t$. Furthermore,
\begin{equation}\label{dbart-complex-forms}
\begin{array}{lll}
\delbar_t\Phi_1^{1,0}(t)=0, &  \delbar_t\Phi_2^{1,0}(t)=0, & \delbar_t\Phi_3^{1,0}(t)=2t\Phi^{1,0}_1(t)\wedge\Phi^{0,1}_2(t),\\[10pt]
\delbar_t\Phi_1^{0,1}(t)=0, &  \delbar_t\Phi_2^{0,1}(t)=0, & \delbar_t\Phi_3^{0,1}(t)=0,
\end{array}
\end{equation}
\end{lemma}
\begin{proof} (I)
By the Kodaira and Spencer theory of small deformations of complex structures, 
$$
\{\varphi^j-\varphi_t(\varphi^j)\,\,\vert \,\,j=1,2,3\}
$$
is a 
coframe of 
$(1,0)$-forms on $N_t$, for $t\in \mathbb{B}(0,\varepsilon)$ (see e.g., \cite[p.75]{GHJ}). Therefore,
$$\begin{array}{l}
\varphi^1-\varphi_t(\varphi^1)= dz_1=:\Phi_1^{1,0}(t), \\[10pt]
\varphi^2-\varphi_t(\varphi^2)= e^{-z_1}dz_2=:\Phi_2^{1,0}(t),\\[10pt]
\varphi^3-\varphi_t(\varphi^3)= e^{z_1}dz_3-te^{-z_1}d\overline{z}_2=:\Phi_3^{1,0}(t)
\end{array}
$${t-complex-forms}
is a complex $(1,0)$-coframe on $N_t$. It is immediate to check that
$$
\Phi_1^{0,1}(t)=\overline{\Phi_1^{1,0}(t)},\quad\Phi_2^{0,1}(t)=e^{\overline{z}_1-z_1}\overline{\Phi_2^{1,0}(t)},\quad
\Phi_3^{0,1}(t)=e^{-\overline{z}_1+z_1}\overline{\Phi_3^{1,0}(t)}.
$$
(II) The proof of \ref{dbart-complex-forms} is a straightforward computation. 
\end{proof}
\begin{lemma}\label{symmetric-symplectic}
The following $2$-form on $N_t$
$$ 
\omega_t:=\frac{i}{2}\left(\Phi_1^{1,0}(t)\wedge\overline{\Phi_1^{1,0}(t)}\right)+\frac12\left(\Phi_2^{0,1}(t)\wedge \Phi_3^{1,0}(t)
+\overline{\Phi_2^{0,1}(t)}\wedge\overline{\Phi_3^{1,0}(t)}\right)
$$
defines a $J_t$-symmetric symplectic structure on $N_t$.
\end{lemma}
\begin{proof}
By definition, $\omega_t$ is $(1,1)$-form with respect to $J_t$ and real. We have
$$\begin{array}{ll}
\omega_t&=\frac{i}{2}(dz_1\wedge d\overline{z}_1)+\frac12e^{-z_1}d\overline{z}_2\wedge(e^{z_1}dz_3-te^{-z_1}d\overline{z}_2)+
\frac12e^{-\overline{z}_1}dz_2\wedge(e^{\overline{z}_1}d\overline{z}_3-\overline{t}e^{-\overline{z}_1}dz_2)\\[10pt]
&=\frac{i}{2}(dz_1\wedge d\overline{z}_1) +\frac12(d\overline{z}_2\wedge dz_3+ dz_2\wedge d\overline{z}_3).
\end{array}
$$
Hence,
$\omega_t^3\neq 0$ and $d\omega_t=0$. 
\end{proof}
By the previous Lemma, $\omega_t=\omega$.
\begin{lemma}\label{massey}
There exists a non vanishing Dolbeault Massey product on $N_t$, for $t\neq 0$, $t\in\mathbb{B}(0,\varepsilon)$. 
\end{lemma}
\begin{proof}
Consider the following Dolbeault classes on $N_t$ defines respectively as
$$a=[2t\Phi^{1,0}_1(t)],\quad b=[\Phi^{0,1}_2(t)], \quad c=[\Phi^{0,1}_2(t)].
$$
Then, $a\cdot b=0, \,b\cdot c=0$. Indeed, 
$$
a\cdot b=[2t\Phi^{1,0}_1(t)\wedge\Phi^{0,1}_2(t)]=[\delbar_t\Phi^{1,0}_3(t)],\quad
b\cdot c=[\Phi^{0,1}_2(t)\wedge \Phi^{0,1}_2(t)]=[0].
$$
Therefore, the Dolbeault triple product $\langle a, b, c\rangle$ is given by
$$
\langle a, b, c\rangle =[\Phi^{1,0}_3(t)\wedge \Phi^{0,1}_2(t)]\in \frac{H^{1,1}_{\delbar_t}(N_t)}{H^{1,0}_{\delbar_t}(N_t)\cdot H^{0,1}_{\delbar_t}(N_t)}
$$
A direct computation shows that $\Phi^{1,0}_3(t)\wedge \Phi^{0,1}_2(t)$ is $\Box^{g_t}_{\delbar_t}$-harmonic, where 
$$
g_t=\sum_{j=1}^3 \Phi^{1,0}_j(t)\otimes \overline{\Phi^{1,0}_j(t)};
$$
consequently, the Dolbeault class $[\Phi^{1,0}_3(t)\wedge \Phi^{0,1}_2(t)]$ does not vanish in $H^{1,1}_{\delbar_t}(N_t)$. \newline 
Let us show that 
$[\Phi^{1,0}_3(t)\wedge \Phi^{0,1}_2(t)]\notin H^{1,0}_{\delbar_t}(N_t)\cdot H^{0,1}_{\delbar_t}(N_t)$.
Set
$$
C_t^{\bullet,\bullet}:=\Lambda^{\bullet,\bullet}(\Span_\C\langle \Phi^{1,0}_1(t),\Phi^{1,0}_2(t),\Phi^{1,0}_3(t)\rangle
\oplus \Span_\C\langle \Phi^{0,1}_1(t),\Phi^{0,1}_2(t),\Phi^{0,1}_3(t)\rangle);
$$
then $C_t^{\bullet,\bullet}$ satisfies the assumptions of \cite[Theorem 1]{AK}. Consequently,
$$
H^{\bullet,\bullet}_{\delbar_t}(C_t^{\bullet,\bullet})\simeq H^{\bullet,\bullet}_{\delbar_t}(N_t).
$$
Explicitly, 
$$
H^{1,0}_{\delbar_t}(N_t)\simeq \Span_\C\langle \Phi^{1,0}_1(t),\Phi^{1,0}_2(t)\rangle,\quad
H^{0,1}_{\delbar_t}(N_t)\simeq \Span_\C\langle \Phi^{0,1}_1(t),\Phi^{0,1}_2(t),\Phi^{0,1}_3(t)\rangle,
$$
and all the representatives are Dolbeault harmonic with respect to the Hermitian metric $g_t$. Therefore, 
$[\Phi^{1,0}_3(t)\wedge \Phi^{0,1}_2(t)]\notin H^{1,0}_{\delbar_t}(N_t)\cdot H^{0,1}_{\delbar_t}(N_t)$ and
$\langle a, b, c\rangle\neq 0$.
\end{proof}
\begin{lemma}\label{non-HLC}
If $t\neq 0$, $t\in\mathbb{B}(0,\varepsilon)$, then $(N,J_t,\omega_t)$ does not satisfy the $\overline{\partial}_t\,\overline{\partial}_t^{\Lambda}$-Lemma.
\end{lemma}
\begin{proof}
Let $\eta$ be the $J_t$-$(0,1)$-form on $N_t$ defined by $\eta:=\Phi^{0,1}_2(t)$. Then $\eta$ is $\Box^{g_t}_{\delbar_t}$-harmonic. Let us compute $\omega^2_t\wedge\eta$.
We immediately get:
\begin{eqnarray*}
\omega^2_t&=&\frac{i}{2}\left(\Phi_1^{1,0}(t)\wedge\overline{\Phi_1^{1,0}(t)}\wedge\Phi_2^{0,1}(t)
\wedge\Phi_3^{1,0}(t)+\Phi_1^{1,0}(t)\wedge\overline{\Phi_1^{1,0}(t)}\wedge\overline{\Phi_2^{0,1}(t)}\wedge
\overline{\Phi_3^{1,0}(t)}\,\right)\\[10pt]
&{}& + \ \frac12 \Phi_2^{0,1}(t)\wedge\Phi_3^{1,0}(t)\wedge\overline{\Phi_2^{0,1}(t)}\wedge
\overline{\Phi_3^{1,0}(t)}.
\end{eqnarray*}
Therefore,
$$
\omega^2_t\wedge\eta=\omega^2_t\wedge\Phi^{0,1}_2(t)=
-\frac{i}{2}\Phi_1^{1,0}(t)\wedge\Phi_2^{0,1}(t)\wedge\Phi_1^{0,1}(t)\wedge\overline{\Phi_2^{0,1}(t)}\wedge
\overline{\Phi_3^{1,0}(t)}
$$
For $t\neq 0$, in view of \eqref{dbart-complex-forms} and \eqref{t-complex-forms}, we get:
$$
\begin{array}{lll}
\delbar_t\overline{\Phi_3^{1,0}(t)}&=&\Phi_1^{0,1}(t)\wedge\overline{\Phi_3^{1,0}(t)},\qquad
\delbar_t\overline{\Phi_2^{0,1}(t)}=-\Phi_1^{0,1}(t)\wedge\overline{\Phi_2^{0,1}(t)}\\[10pt]
\frac{1}{2t}\delbar_t\Phi_3^{1,0}(t)&=&\Phi_1^{1,0}(t)\wedge\Phi_2^{0,1}(t).
\end{array}
$$
Thus:
$$
\omega^2_t\wedge\eta=-\frac{i}{2t}\delbar_t\left(\Phi_3^{1,0}(t)\wedge\Phi_1^{0,1}(t)\wedge\overline{\Phi_2^{0,1}(t)}\wedge
\overline{\Phi_3^{1,0}(t)}\,\right),
$$
that is the Dolbeault class $[\omega^2_t\wedge\eta]$ vanishes in $H^{2,3}_{\delbar_t}(N_t)$ for 
$t\neq 0$, $t\in\mathbb{B}(0,\varepsilon)$. Therefore $(N,\omega_t)$ does not satisfy HLC and consequently $(N, J_t,\omega_t)$ does not satisfy the $\overline{\partial}_t\,\overline{\partial}_t^{\Lambda}$-Lemma.
\end{proof}
Summimg up, we have proved the following:
\begin{theorem}\label{7.6}
Let $N$ be the be the differentiable manifold underlying the Nakamura manifold $\Gamma\backslash \C^3$. Then there exists a $1$-parameter complex family 
of complex structures $J_t$ on $N$ and a symplectic structure $\omega$, for $t\in\mathbb{B}(0,\varepsilon)$ such that, 
\begin{itemize}
 \item $J_0=J_M$.
 \item $(N,J_N,\omega)$ satifies the $\overline{\partial}\,\overline{\partial}^\Lambda$-Lemma and the complex manifold $(N,J_N)$ is geometrically Dolbeault 
 formal.
 \item For $t\in\mathbb{B}(0,\varepsilon), t\neq 0$, $(N,J_t,\omega_t)$ does not satifies the $\overline{\partial}_t\,\overline{\partial}_t^{\Lambda}$-Lemma and it is not Dolbeault 
 formal.
\end{itemize}
\end{theorem}
As a corollary, we obtain the following
\begin{theorem}\label{7.7}
The $\overline{\partial}\,\overline{\partial}^\Lambda$-Lemma is an unstable property under small deformations of the complex structure.
\end{theorem}

\subsection{Proof of Theorem C}

Theorems \ref{7.6} and \ref{7.7} give the proof of (3). Now it is enough to prove (1) and (2). 

\begin{proof}[Proof of Theorem C $(1)$] Recall that, $\omega$ satisfies the Hard Lefschetz Condition if and only if for every $0\leq k\leq n$,
$$
[\omega^k] : H^{n-k}_{\dbar}(X) \to H^{n+k}_{\dbar}(X)
$$
is an isomorphism, which is equivalent to that the determinant, say $\det\, [\omega^k]$, of the above map is non-zero. Since $\det\, [\omega^k] $ depends smoothly on $\omega$, Theorem C $(1)$ follows.
\end{proof}

\begin{proof}[Proof of Theorem C $(2)$] Let $\dbar_t$ be a smooth family of complex structures on $X_t:=(X,J_t)$. By \cite[Theorem A, B]{AT13}, we know that if $X=(X,J_0)$ satisfies the $\partial_0\dbar_0$-Lemma then
$$
\sum_{p+q=k} {\rm dim}\, H^{p,q}_{\dbar_0}(X) ={\rm dim}\, H^k_{d}(X).
$$
On the other hand, by Fr\"olicher's theorem, we always have
$$
\sum_{p+q=k} {\rm dim}\, H^{p,q}_{\dbar_t}(X_t) \geq {\rm dim}\, H^k_{d}(X).
$$
But notice that ${\rm dim}\, H^k_{d}(X)$ does not depend on $t$ and ${\rm dim}\, H^{p,q}_{\dbar_t}(X_t)$ is upper semi-continuous,  so $ {\rm dim}\, H^{p,q}_{\dbar_t}(X_t)$ does not depend on $t$, which also implies that
$$
{\rm dim}\, H^{p,q}_{\dbar_t}(X_t)={\rm dim}\, H^{n-q,n-p}_{\dbar^\Lambda_t}(X_t)
$$
does not depend on $t$. Assume further that $(X, \omega)$ satisfies the $\dbar_0\dbar_0^\Lambda$-Lemma, then Theorem \ref{d-dbar} gives
$$
\dim H^{p,q}_{BC}(X)+ \dim H^{p,q}_{A}(X) = \dim H_{\dbar_0}^{p,q}(X)+ \dim H_{\dbar_0^\Lambda}^{p,q}(X)
$$
and, 
$$
\dim H^{p,q}_{BC}(X_t)+ \dim H^{p,q}_{A}(X_t) \geq \dim H_{\dbar_t}^{p,q}(X_t)+ \dim H_{\dbar_t^\Lambda}^{p,q}(X_t).
$$
Therefore, by the upper semicontinuity of $t\mapsto\dim H^{p,q}_{BC}(X_t)$ and $t\mapsto\dim H^{p,q}_{A}(X_t)$, we obtain
\begin{eqnarray*}
 \dim H^{p,q}_{BC}(X_0)+ \dim H^{p,q}_{A}(X_0)&\geq& \dim H^{p,q}_{BC}(X_t)+ \dim H^{p,q}_{A}(X_t)\\
                                              &\geq& \dim H^{p,q}_{\dbar}(X_t)+ \dim H^{p,q}_{\dbar^\Lambda}(X_t)\\
                                              &=&\dim H^{p,q}_{\dbar}(X_0)+ \dim H^{p,q}_{\dbar^\Lambda}(X_0)\\
                                              &=& \dim H^{p,q}_{BC}(X_0)+ \dim H^{p,q}_{A}(X_0),
\end{eqnarray*}
that is 
$$
\dim H^{p,q}_{BC}(X_t)+ \dim H^{p,q}_{A}(X_t)= \dim H^{p,q}_{\dbar}(X_t)+ \dim H^{p,q}_{\dbar^\Lambda}(X_t).
$$
Hence, by Theorem \ref{d-dbar}, $(X, J_t, \omega)$ also satisfies the $\dbar_t\dbar_t^\Lambda$-Lemma. By \cite[Corollary 3.7]{AT13},  we also know that $X$ satisfies the   $\partial_t\dbar_t$-Lemma. Thus satisfying both the $\partial\dbar$-Lemma and the $\dbar\,\dbar^\Lambda$-Lemma is a stable property under small deformations of the complex structure.
\end{proof}

\end{document}